\documentclass[]{amsart}

\usepackage[english]{babel}
\usepackage[utf8]{inputenc}
\usepackage{paralist}
\usepackage{amsmath, amsfonts, amssymb, amsthm}
\usepackage{color}
\usepackage[normalem]{ulem}
\usepackage{tikz}

\newcommand{\CC}{\mathbb{C}}

\newcommand{\NN}{\mathbb{N}}

\newcommand{\QQ}{\mathbb{Q}}

\newcommand{\ZZ}{\mathbb{Z}}

\newcommand{\set}[1]{\left\{ #1 \right\}}

\newcommand{\red}{_{\mathrm{red}}}

\newtheorem{mymasterthm}{notForUse}
\theoremstyle{definition}

\theoremstyle{plain}
\newtheorem{mylemma}[mymasterthm]{Lemma}
\newtheorem{mythm}[mymasterthm]{Theorem}

\allowdisplaybreaks

\title[Norm Form Equations with Solutions in Multi-Recurrences]{Norm Form Equations with Solutions Taking Values in a Multi-Recurrence}

\author[C. Fuchs]{Clemens Fuchs}
\author[S. Heintze]{Sebastian Heintze}
\thanks{Supported by Austrian Science Fund (FWF): I4406.}

\address{University of Salzburg\newline
	\indent Department of Mathematics\newline
	\indent Hellbrunnerstr. 34 \newline
	\indent A-5020 Salzburg, Austria}
\email{clemens.fuchs@sbg.ac.at, sebastian.heintze@sbg.ac.at}

\begin{document}
	
	\maketitle
	
	\begin{abstract}
		We are interested in solutions of a norm form equation that takes values in a given multi-recurrence. We show that among the solutions there are only finitely many values in each component which lie in the given multi-recurrence unless the recurrence is of precisely described exceptional shape. This gives a variant of the question on arithmetic progressions in the solution set of norm form equations.
	\end{abstract}
	
	\section{Introduction}
	
	Let $ K $ be an algebraic number field of degree $ d $ and let $ \alpha_1, \ldots, \alpha_n $ be linearly independent elements of $ K $ over $ \QQ $. Thus we have $ n \leq d $.
	We denote by $ N_{K/\QQ} $ the field norm of $ K $ and for an integer $ m $ we consider the norm form equation given by
	\begin{equation}
		\label{p7-eq:normformeq}
		N_{K/\QQ}( x_1 \alpha_1 + \cdots + x_n \alpha_n ) = m.
	\end{equation}
	It was proved by Schmidt (cf. \cite{schmidt-1971}) that if the $ \ZZ $-module generated by $ \alpha_1, \ldots, \alpha_n $ contains a submodule which is a full module in a subfield of $ \QQ(\alpha_1, \ldots, \alpha_n) $ different from the imaginary quadratic fields and $ \QQ $, then this equation has infinitely many solutions $ (x_1, \ldots, x_n) \in \ZZ^n $ for some $ m $.
	
	Observe that by multiplying with a common denominator we may assume that $ \alpha_1, \ldots, \alpha_n $ are algebraic integers and we will assume this from now on. In the present paper we are interested in those $ x_i $ which can be in the value set of a multi-recurrence. Since we are interested in proving finiteness results, we will always assume that \eqref{p7-eq:normformeq} has infinitely many solutions. This immediately implies that there exists an index $ i $ such that we have infinitely many solutions of \eqref{p7-eq:normformeq} with different $ x_i $.
	
	\section{Notation and Results}
	
	Let $ s $ be an arbitrary integer and $ G $ a multi-recurrence, that is a mapping from $ \ZZ^s \rightarrow \CC $ which is given by a function of polynomial-exponential type
	\begin{equation}
		\label{p7-eq:multirecurr}
		G(k_1,\ldots,k_s) = \sum_{j=1}^{q} P_j(k_1,\ldots,k_s) \alpha_{j1}^{k_1} \cdots \alpha_{js}^{k_s}
	\end{equation}
	with $ P_j(X_1,\ldots,X_s) \in \CC[X_1,\ldots,X_s] $ and non-zero $ \alpha_{j1},\ldots,\alpha_{js} \in \CC $ for $ j = 1,\ldots,q $.
	It is well-known that for such multi-sequences the value $ G(k_1,\ldots,k_s) $ can be described by certain linear combinations of values of $ G $ with shifted entries and hence they are the natural extension of sequences which satisfy a linear recurrence relation.
	The multi-recurrence $ G $ is called simple if $ \deg P_j = 0 $ for all $ j = 1,\ldots,q $; we put $ P_j(X_1,\ldots,X_s) = p_j $ in this case.
	Moreover, we say that $ G $ is defined over a number field $ K $ if the $ \alpha_{ji} $ and the coefficients of the $ P_j $ are elements of $ K $ for all $ j $ and $ i $.
	
	Clearly, $ G $ can have infinitely many zeros $ (k_1,\ldots,k_s) \in \ZZ^s $. We mention that the structure of solutions $ (k_1,\ldots,k_s) $ of $ G(k_1,\ldots,k_s) = 0 $, in the case that there are infinitely many of them, is not known (in contrast to the case $ s = 1 $ of linear recurring sequences where the Skolem-Mahler-Lech theorem says that all solutions lie in a finite union of arithmetic progressions).
	For more details on this see e.g. \cite{schmidt-2003}.
	
	It is also well-known and not too hard to prove (see e.g. the proof of Theorem 2.1 in \cite{berczes-hajdu-petho-2010}) that each component of the solutions of \eqref{p7-eq:normformeq} is contained in the union of finitely many multi-recurrences. In fact we have for each $ \ell \in \set{1,\ldots,n} $ that
	\begin{equation*}
		x_{\ell} = H(h_1,\ldots,h_r)
	\end{equation*}
	for a multi-recurrence $ H $ from a finite set of multi-recurrences all having the form
	\begin{equation*}
		H(h_1,\ldots,h_r) := \sum_{i=1}^{n} \tau_i \sigma_i(\varepsilon_1)^{h_1} \cdots \sigma_i(\varepsilon_r)^{h_r}
	\end{equation*}
	where the $ \tau_i $ are constants, $ \varepsilon_1, \ldots, \varepsilon_r $ is a system of fundamental units in the ring of integers and $ \sigma_1, \ldots, \sigma_n $ are embeddings of $ K $ in $ \CC $ such that the matrix $ (\sigma_i(\alpha_j)) $ has non-zero determinant.
	For the sake of completeness we give a short sketch of this proof. Let $ \sigma_1, \ldots, \sigma_d $ be the elements of the Galois group $ \mathrm{Gal}(K/\QQ) $. For $ i = 1,\ldots,d $ we get
	\begin{equation*}
		x_1 \sigma_i(\alpha_1) + \cdots + x_n \sigma_i(\alpha_n) = \sigma_i(\varepsilon) \sigma_i(\mu)
	\end{equation*}
	where $ \varepsilon $ is a unit and $ \mu $ an element of norm $ m $ which can be chosen from a finite set by Lemma 4 in \cite{evertse-gyory-1997}.
	Let us therefore consider a fixed value of $ \mu $. Choose the order of the isomorphisms $ \sigma_1, \ldots, \sigma_d $ in such a way that the matrix
	\begin{equation*}
		M =
		\begin{pmatrix}
			\sigma_1(\alpha_1) & \cdots & \sigma_1(\alpha_n) \\
			\vdots & \ddots & \vdots \\
			\sigma_n(\alpha_1) & \cdots & \sigma_n(\alpha_n)
		\end{pmatrix}
	\end{equation*}
	has non-zero determinant. This implies
	\begin{equation*}
		\begin{pmatrix}
			x_1 \\ \vdots \\ x_n
		\end{pmatrix}
		= M^{-1}
		\begin{pmatrix}
			\sigma_1(\varepsilon) \sigma_1(\mu) \\ \vdots \\ \sigma_n(\varepsilon) \sigma_n(\mu)
		\end{pmatrix}
	\end{equation*}
	and we get
	\begin{equation*}
		x_{\ell} = \sum_{i=1}^{n} m_{\ell i} \sigma_i(\varepsilon) \sigma_i(\mu).
	\end{equation*}
	Applying Dirichlet's unit theorem gives the above statement.
	
	In the sequel we call the set $ \Lambda(A,b) := \set{(k_1,\ldots,k_s)A+b : k_1,\ldots,k_s \in \ZZ} $ with $ A $ an $ (s \times r) $-matrix with entries in $ \ZZ $ and $ b $ a row vector with $ r $ entries from $ \ZZ $ a \emph{shifted sublattice} of $ \ZZ^r $. We say that $ (h_1,\ldots,h_r) $ runs through a shifted sublattice of $ \ZZ^r $ if $ (h_1,\ldots,h_r) \in \Lambda(A,b) $ for some $ A \in \ZZ^{s \times r} $ and $ b \in \ZZ^r $.
	
	Let $ L $ be a finite extension of $ K $. Then we can lift equation \eqref{p7-eq:normformeq} to a norm form equation in $ L $ using the tower formula for the field norm
	\begin{align*}
		N_{L/\QQ}( x_1 \alpha_1 + \cdots + x_n \alpha_n ) &= N_{K/\QQ}(N_{L/K}( x_1 \alpha_1 + \cdots + x_n \alpha_n )) \\
		&= N_{K/\QQ}(( x_1 \alpha_1 + \cdots + x_n \alpha_n )^{[L:K]}) = m^{[L:K]}.
	\end{align*}
	We call
	\begin{equation}
		\label{p7-eq:liftnormform}
		N_{L/\QQ}( x_1 \alpha_1 + \cdots + x_n \alpha_n ) = m^{[L:K]}
	\end{equation}
	the \emph{lifted norm form equation}.
	It is clear that all solutions of \eqref{p7-eq:normformeq} are solutions of \eqref{p7-eq:liftnormform} as well.
	
	Our interest applies to solutions $ (x_1,\ldots,x_n) $ of a norm form equation with the property that $ x_{\ell} = G(k_1,\ldots,k_s) $ for some $ \ell $ and a given multi-recurrence $ G $.
	It is easy to see that this problem may have infinitely many solutions, e.g. this is the case if
	\begin{equation}
		\label{p7-eq:unavoidexcepgen}
		G(k_1,\ldots,k_s) = H(h_1,\ldots,h_r)\left|_{(h_1,\ldots,h_r) = (k_1,\ldots,k_s)A+b} \right. + G_0(k_1,\ldots,k_s)
	\end{equation}
	for all $ (k_1,\ldots,k_s) $ within an arithmetic progression of $ s $-dimensional vectors (i.e. the cartesian product of $ s $ arithmetic progressions of integers), where $ H $ is a multi-recurrence coming from the solutions of the lifted norm form equation, evaluated at points that run through a shifted sublattice, and $ G_0 $ is a multi-recurrence that has infinitely many zeros along the intersection of that sublattice and the arithmetic progression.
	We call this case an \emph{unavoidable exception}.
	If $ G $ has the form \eqref{p7-eq:unavoidexcepgen} of an unavoidable exception with the additional property that $ G_0 = 0 $ and that $ H $ comes directly from the solutions of \eqref{p7-eq:normformeq}, i.e. trivially lifted to itself ($ L=K $), then we call it a \emph{reduced unavoidable exception}.
	
	We have the following first theroem that describes the situation in the general case:
	\begin{mythm}
		\label{p7-thm:generalcase}
		Let $ G $ be a simple multi-recurrence that is defined over $ K $ by \eqref{p7-eq:multirecurr} where all the $ \alpha_{ji} $ are algebraic integers. Then for any $ \ell \in \set{1,\ldots,n} $ there are at most finitely many values of $ x_{\ell} $ such that $ (x_1,\ldots,x_n) \in \ZZ^n $ is a solution of \eqref{p7-eq:normformeq} and
		\begin{equation*}
			x_{\ell} = G(k_1,\ldots,k_s)
		\end{equation*}
		for suitable $ (k_1,\ldots,k_s) \in \NN^s $ unless $ G $ has the form of an unavoidable exception.
	\end{mythm}
	
	In the special case $ s = 1 $ of a linear recurrence sequence we can prove the following stronger result:
	\begin{mythm}
		\label{p7-thm:specialcase}
		Let $ G $ be a simple linear recurrence sequence that is defined over $ K $ by \eqref{p7-eq:multirecurr} with $ s = 1 $. Then for any $ \ell \in \set{1,\ldots,n} $ there are at most finitely many values of $ x_{\ell} $ such that $ (x_1,\ldots,x_n) \in \ZZ^n $ is a solution of \eqref{p7-eq:normformeq} and
		\begin{equation*}
			x_{\ell} = G(k_1)
		\end{equation*}
		for suitable $ k_1 \in \NN $ unless $ G $ has the form of a reduced unavoidable exception.
	\end{mythm}
	
	\section{Preliminaries}
	
	Let $ F $ be an algebraically closed field of characteristic $ 0 $. Denote the multiplicative group of non-zero elements by $ F^* $ and let $ (F^*)^n $ be the direct product consisting of $ n $-tuples $ (y_1,\ldots,y_n) $ with $ y_{\ell} \in F^* $ for $ \ell = 1,\ldots,n $ equipped with component-wise multiplication.
	Let $ \Gamma $ be a subgroup of $ (F^*)^n $ and suppose that $ (a_1,\ldots,a_n) \in (F^*)^n $. We will consider the generalized unit equation
	\begin{equation}
		\label{p7-eq:genuniteq}
		a_1 y_1 + \cdots + a_n y_n = 1
	\end{equation}
	in $ (y_1,\ldots,y_n) \in \Gamma $.
	A solution $ (y_1,\ldots,y_n) $ is called non-degenerate if no subsum of the left hand side of \eqref{p7-eq:genuniteq} vanishes, which means that $ \sum_{i \in I} a_i y_i \neq 0 $ for any non-empty subset $ I $ of $ \set{1,\ldots,n} $.
	The following lemma proved by Evertse, Schlickewei and Schmidt as Theorem 1.1 in \cite{evertse-schlickewei-schmidt-2002} will be used in our proofs:
	\begin{mylemma}
		\label{p7-lemma:boundnondegensol}
		Suppose that $ \Gamma $ has finite rank $ r $. Then the number of non-degenerate solutions $ (y_1,\ldots,y_n) \in \Gamma $ of equation \eqref{p7-eq:genuniteq} is bounded by
		\begin{equation*}
			\exp \left( (6n)^{3n} (r+1) \right).
		\end{equation*}
	\end{mylemma}
	In particular this implies that there are only finitely many non-degenerate solutions of the generalized unit equation.
	
	\section{Proofs}
	
	Now we are going to prove our two theorems. We start with the general case which will also be the base for the special one.
	
	\begin{proof}[Proof of Theorem \ref{p7-thm:generalcase}]
		Assume that for some $ \ell \in \set{1,\ldots,n} $ there are infinitely many values of $ x_{\ell} $ such that $ (x_1,\ldots,x_n) \in \ZZ^n $ is a solution of \eqref{p7-eq:normformeq} and
		\begin{equation*}
			x_{\ell} = G(k_1,\ldots,k_s)
		\end{equation*}
		for suitable $ (k_1,\ldots,k_s) \in \NN^s $.
		Then we can choose an infinite sequence of such values for $ x_{\ell} $ that are all non-zero and pairwise distinct.
		Each value corresponds to another vector $ (k_1,\ldots,k_s) $ satisfying $ x_{\ell} = G(k_1,\ldots,k_s) $. This vector is not necessarily uniquely determined, but we will fix one possible vector for each value of $ x_{\ell} $ now.
		
		Thus we get a sequence of vectors $ (k_1,\ldots,k_s) $. If the first component $ k_1 $ takes a fixed value for infinitely many elements of our sequence, then we go to a subsequence where $ k_1 $ is constant. Otherwise we go to a subsequence where $ k_1 $ is strictly increasing and non-zero.
		We perform the same procedure with the other components. After reindexing we can assume that $ k_1,\ldots,k_t $ are strictly increasing and that $ k_{t+1},\ldots,k_s $ are constant.
		
		The next step is to analyse whether there are linear dependencies between the non-constant components $ k_1,\ldots,k_t $ of the vectors $ (k_1,\ldots,k_s) $.
		If
		\begin{equation*}
			a_0^{(t)} + a_1^{(t)} k_1 + \cdots + a_t^{(t)} k_t = 0
		\end{equation*}
		for infinitely many vectors $ (k_1,\ldots,k_s) $ and constant rational integers $ a_0^{(t)}, \ldots, a_t^{(t)} $ which are not all zero, then we go to a subsequence where this equation holds.
		By reindexing we can assume that $ a_t^{(t)} \neq 0 $ and get
		\begin{equation*}
			k_t = \frac{1}{a_t^{(t)}} \left( -a_0^{(t)} - a_1^{(t)} k_1 - \cdots - a_{t-1}^{(t)} k_{t-1} \right).
		\end{equation*}
		If in addition for infinitely many vectors $ (k_1,\ldots,k_s) $ and constant rational integers $ a_0^{(t-1)}, \ldots, a_{t-1}^{(t-1)} $ which are not all zero the equation
		\begin{equation*}
			a_0^{(t-1)} + a_1^{(t-1)} k_1 + \cdots + a_{t-1}^{(t-1)} k_{t-1} = 0
		\end{equation*}
		holds, then we perform the analogous procedure.
		
		Thus we can assume that there is an index $ l \leq t $ such that both
		\begin{equation*}
			a_0^{(l)} + a_1^{(l)} k_1 + \cdots + a_l^{(l)} k_l = 0
		\end{equation*}
		for infinitely many vectors $ (k_1,\ldots,k_s) $ and constant rational integers $ a_0^{(l)}, \ldots, a_l^{(l)} $ implies $ a_0^{(l)} = \cdots = a_l^{(l)} = 0 $, and
		\begin{align}
			k_{l+1} &= \frac{1}{d_{l+1}} \left( \widetilde{a_0}^{(l+1)} + \widetilde{a_1}^{(l+1)} k_1 + \cdots + \widetilde{a_l}^{(l+1)} k_l \right) \nonumber \\
			\label{p7-eq:krepresentation}
			& \vdots \\
			k_t &= \frac{1}{d_t} \left( \widetilde{a_0}^{(t)} + \widetilde{a_1}^{(t)} k_1 + \cdots + \widetilde{a_l}^{(t)} k_l \right) \nonumber
		\end{align}
		with rational integers $ d_j > 0 $ and $ \widetilde{a_i}^{(j)} $ for our infinite sequence.
		Therefore we define the finite extension $ L $ of $ K $ as
		\begin{equation*}
			L = K \left( \set{\sqrt[d_j]{\alpha_{\kappa \nu}} : j = l+1,\ldots,t;\ \kappa = 1,\ldots,q;\ \nu = 1,\ldots,s} \right).
		\end{equation*}
		
		Now we mention that our sequence of values of $ x_{\ell} $ corresponds to a sequence of vectors $ (h_1,\ldots,h_r) $ satisfying
		\begin{equation*}
			x_{\ell} = H(h_1,\ldots,h_r)
		\end{equation*}
		for a fixed multi-recurrrence $ H $ coming from the solutions of the lifted norm form equation \eqref{p7-eq:liftnormform} if we go to a subsequence once again.
		Moreover, the sequence of values of $ x_{\ell} $ corresponds to a sequence of vectors $ (h_1',\ldots,h_{r'}') $ satisfying
		\begin{equation*}
			x_{\ell} = H'(h_1',\ldots,h_{r'}')
		\end{equation*}
		for a fixed multi-recurrrence $ H' $ coming from the solutions of the norm form equation \eqref{p7-eq:normformeq} if we go to a subsequence once again.
		
		Altogether we have the following correspondences which will be used tacitly in the sequel:
		\begin{equation*}
			x_{\ell} \leftrightarrow (k_1,\ldots,k_s) \leftrightarrow (h_1,\ldots,h_r) \leftrightarrow (h_1',\ldots,h_{r'}').
		\end{equation*}
		When we say that something holds for infinitely many vectors, this means for infinitely many vectors in our sequence and we implicitely go to a subsequence where this property is satisfied by all elements.
		
		Let us take a closer look at the multi-recurrence $ G $.
		Suppose that there is a constant $ c_{ij} $ for some distinct indices $ i,j $ such that for infinitely many vectors the equation
		\begin{equation}
			\label{p7-eq:reductionofG}
			\alpha_{j1}^{k_1} \cdots \alpha_{js}^{k_s} = c_{ij} \alpha_{i1}^{k_1} \cdots \alpha_{is}^{k_s}
		\end{equation}
		holds.
		Then we can construct a multi-recurrence
		\begin{equation*}
			G_0^{(I)}(k_1,\ldots,k_s) = \sum \left( p_j \alpha_{j1}^{k_1} \cdots \alpha_{js}^{k_s} - p_j c_{ij} \alpha_{i1}^{k_1} \cdots \alpha_{is}^{k_s} \right)
		\end{equation*}
		which is zero for the vectors in our sequence such that
		\begin{align*}
			G\red (k_1,\ldots,k_s) :&= G(k_1,\ldots,k_s) - G_0^{(I)}(k_1,\ldots,k_s) \\
			&= \sum_{j=1}^{\widetilde{q}} \widetilde{p_j} \alpha_{j1}^{k_1} \cdots \alpha_{js}^{k_s}
		\end{align*}
		does not contain two summands which satisfy a relation of the shape \eqref{p7-eq:reductionofG}.
		We perform the analogous procedure for the multi-recurrences $ H $ and $ H' $ to get
		\begin{align*}
			H\red (h_1,\ldots,h_r) :&= H(h_1,\ldots,h_r) - H_0(h_1,\ldots,h_r) \\
			&= \sum_{i=1}^{\widetilde{n}} \widetilde{\tau_i} \sigma_i(\varepsilon_1)^{h_1} \cdots \sigma_i(\varepsilon_r)^{h_r}
		\end{align*}
		and
		\begin{align*}
			H'\red (h_1,\ldots,h_r) :&= H'(h_1,\ldots,h_r) - H_0'(h_1,\ldots,h_r) \\
			&= \sum_{i=1}^{\widetilde{n}'} \widetilde{\tau_i}' \sigma_i'(\varepsilon_1')^{h_1'} \cdots \sigma_i'(\varepsilon_{r'}')^{h_{r'}'}.
		\end{align*}
		
		For our sequence we have the equality
		\begin{align*}
			H\red (h_1,\ldots,h_r) &= H(h_1,\ldots,h_r) - H_0(h_1,\ldots,h_r) \\
			&= H(h_1,\ldots,h_r) = x_{\ell} = G(k_1,\ldots,k_s) \\
			&= G(k_1,\ldots,k_s) - G_0^{(I)}(k_1,\ldots,k_s) \\
			&= G\red (k_1,\ldots,k_s).
		\end{align*}
		Putting in the sum representations, this yields
		\begin{equation}
			\label{p7-eq:applylemma}
			\sum_{i=1}^{\widetilde{n}} \widetilde{\tau_i} \sigma_i(\varepsilon_1)^{h_1} \cdots \sigma_i(\varepsilon_r)^{h_r}
			- \sum_{j=1}^{\widetilde{q}} \widetilde{p_j} \alpha_{j1}^{k_1} \cdots \alpha_{js}^{k_s}
			= 0
		\end{equation}
		which can be rewritten as
		\begin{equation*}
			\sum_{i=1}^{\widetilde{n}} \widetilde{\tau_i} \sigma_i(\varepsilon_1)^{h_1} \cdots \sigma_i(\varepsilon_r)^{h_r}
			- \sum_{j=1}^{\widetilde{q}-1} \widetilde{p_j} \alpha_{j1}^{k_1} \cdots \alpha_{js}^{k_s}
			= \widetilde{p_{\widetilde{q}}} \alpha_{\widetilde{q}1}^{k_1} \cdots \alpha_{\widetilde{q}s}^{k_s}
		\end{equation*}
		and after dividing by the right hand side this is
		\begin{equation*}
			\sum_{i=1}^{\widetilde{n}} \frac{\widetilde{\tau_i}}{\widetilde{p_{\widetilde{q}}}} \sigma_i(\varepsilon_1)^{h_1} \cdots \sigma_i(\varepsilon_r)^{h_r} \alpha_{\widetilde{q}1}^{-k_1} \cdots \alpha_{\widetilde{q}s}^{-k_s}
			- \sum_{j=1}^{\widetilde{q}-1} \frac{\widetilde{p_j}}{\widetilde{p_{\widetilde{q}}}} \left( \frac{\alpha_{j1}}{\alpha_{\widetilde{q}1}} \right)^{k_1} \cdots \left( \frac{\alpha_{js}}{\alpha_{\widetilde{q}s}} \right)^{k_s}
			= 1.
		\end{equation*}
		The previous line can be seen as a generalized unit equation in the $ \widetilde{n} + \widetilde{q} - 1 $ unknowns
		\begin{equation}
			\label{p7-eq:unknowns}
			\sigma_i(\varepsilon_1)^{h_1} \cdots \sigma_i(\varepsilon_r)^{h_r} \alpha_{\widetilde{q}1}^{-k_1} \cdots \alpha_{\widetilde{q}s}^{-k_s}
			,\
			\left( \frac{\alpha_{j1}}{\alpha_{\widetilde{q}1}} \right)^{k_1} \cdots \left( \frac{\alpha_{js}}{\alpha_{\widetilde{q}s}} \right)^{k_s}
		\end{equation}
		and by Lemma \ref{p7-lemma:boundnondegensol} either there are only finitely many solutions or we have a vanishing subsum.
		In the first case all expressions \eqref{p7-eq:unknowns} are constant for infinitely many vectors.
		In the second case some expressions \eqref{p7-eq:unknowns} are constant for infinitely many vectors and the remaining terms make up a vanishing subsum.
		We multiply the vanishing subsum by $ \widetilde{p_{\widetilde{q}}} \alpha_{\widetilde{q}1}^{k_1} \cdots \alpha_{\widetilde{q}s}^{k_s} $ and then do the same as we have done with \eqref{p7-eq:applylemma}.
		Since in each step we have less summands in the equation of the form \eqref{p7-eq:applylemma} as in the step before, this procedure ends after finitely many steps.
		
		This gives us a set of equations, valid for infinitely many vectors, of the following three types (in all three types it is $ i \neq j $):
		\begin{align}
			\label{p7-eq:typeA}
			\sigma_i(\varepsilon_1)^{h_1} \cdots \sigma_i(\varepsilon_r)^{h_r} &= C_{ij} \alpha_{j1}^{k_1} \cdots \alpha_{js}^{k_s} \\
			\label{p7-eq:typeB}
			\sigma_i(\varepsilon_1)^{h_1} \cdots \sigma_i(\varepsilon_r)^{h_r} &= D_{ij} \sigma_j(\varepsilon_1)^{h_1} \cdots \sigma_j(\varepsilon_r)^{h_r} \\
			\label{p7-eq:typeC}
			\alpha_{i1}^{k_1} \cdots \alpha_{is}^{k_s} &= E_{ij} \alpha_{j1}^{k_1} \cdots \alpha_{js}^{k_s}.
		\end{align}
		Each expression $ \sigma_i(\varepsilon_1)^{h_1} \cdots \sigma_i(\varepsilon_r)^{h_r} $ for $ i = 1,\ldots,\widetilde{n} $ and each expression $ \alpha_{j1}^{k_1} \cdots \alpha_{js}^{k_s} $ for $ j = 1,\ldots,\widetilde{q} $ occurs at least once among those equations.
		
		By our construction of the multi-recurrences $ G\red $ and $ H\red $ there cannot be an equation of type \eqref{p7-eq:typeB} or \eqref{p7-eq:typeC}.
		Moreover, no expression $ \sigma_i(\varepsilon_1)^{h_1} \cdots \sigma_i(\varepsilon_r)^{h_r} $ for $ i = 1,\ldots,\widetilde{n} $ and no expression $ \alpha_{j1}^{k_1} \cdots \alpha_{js}^{k_s} $ for $ j = 1,\ldots,\widetilde{q} $ can occur more than once among the equations of type \eqref{p7-eq:typeA} since otherwise we could deduce an equation of type \eqref{p7-eq:typeB} or \eqref{p7-eq:typeC}.
		Thus each expression $ \sigma_i(\varepsilon_1)^{h_1} \cdots \sigma_i(\varepsilon_r)^{h_r} $ for $ i = 1,\ldots,\widetilde{n} $ and each expression $ \alpha_{j1}^{k_1} \cdots \alpha_{js}^{k_s} $ for $ j = 1,\ldots,\widetilde{q} $ occurs exactly once among the equations of type \eqref{p7-eq:typeA}.
		Therefore we have $ \widetilde{n} = \widetilde{q} $ and after a suitable reindexing
		\begin{equation*}
			\sigma_i(\varepsilon_1)^{h_1} \cdots \sigma_i(\varepsilon_r)^{h_r} = C_{ii} \alpha_{i1}^{k_1} \cdots \alpha_{is}^{k_s}
		\end{equation*}
		for $ i = 1,\ldots,\widetilde{q} $.
		Since $ k_{t+1},\ldots,k_s $ are constant in our sequence, this can be rewritten as
		\begin{equation}
			\label{p7-eq:exppartequal}
			\sigma_i(\varepsilon_1)^{h_1} \cdots \sigma_i(\varepsilon_r)^{h_r} = \widetilde{C_i} \alpha_{i1}^{k_1} \cdots \alpha_{it}^{k_t}
		\end{equation}
		for $ i = 1,\ldots,\widetilde{q} $.
		
		The same steps we have done for $ H\red $ in the last paragraphs can be done for $ H'\red $ as well.
		So we have $ \widetilde{n}' = \widetilde{q} $ and
		\begin{equation*}
			\sigma_i'(\varepsilon_1')^{h_1'} \cdots \sigma_i'(\varepsilon_{r'}')^{h_{r'}'} = \widetilde{C_i}' \alpha_{i1}^{k_1} \cdots \alpha_{it}^{k_t}
		\end{equation*}
		for $ i = 1,\ldots,\widetilde{q} $.
		Let $ \left( \widehat{k_1},\ldots,\widehat{k_s} \right) $ be the first (smallest) element in our sequence.
		Then we get from the last equation by division the following one:
		\begin{equation}
			\label{p7-eq:alphaisunit}
			\sigma_i'(\varepsilon_1')^{h_1' - \widehat{h_1'}} \cdots \sigma_i'(\varepsilon_{r'}')^{h_{r'}' - \widehat{h_{r'}'}}
			= \alpha_{i1}^{k_1 - \widehat{k_1}} \cdots \alpha_{it}^{k_t - \widehat{k_t}}.
		\end{equation}
		Since the left hand side is a unit in the ring of integers (of $ K $ and thus also of $ L $), the right hand side must be a unit, too.
		Moreover, the exponents $ k_1 - \widehat{k_1}, \ldots, k_t - \widehat{k_t} $ are all positive rational integers by construction and the bases $ \alpha_{i1}, \ldots, \alpha_{it} $ are algebraic integers by assumption.
		Thus, $ \alpha_{i1}, \ldots, \alpha_{it} $ are units in the ring of integers for $ i = 1,\ldots,\widetilde{q} $ (of $ K $ and thus also of $ L $).
		
		From here on we will always work over $ L $.
		We use the representations \eqref{p7-eq:krepresentation} to rewrite equation \eqref{p7-eq:exppartequal} as
		\begin{align*}
			\sigma_i(\varepsilon_1)^{h_1} \cdots \sigma_i(\varepsilon_r)^{h_r} &= \widetilde{C_i} \alpha_{i1}^{k_1} \cdots \alpha_{it}^{k_t} \\
			&= \widetilde{C_i} \alpha_{i1}^{k_1} \cdots \alpha_{il}^{k_l} \prod_{u=l+1}^{t} \alpha_{iu}^{\frac{1}{d_u} \left( \widetilde{a_0}^{(u)} + \widetilde{a_1}^{(u)} k_1 + \cdots + \widetilde{a_l}^{(u)} k_l \right)} \\
			&= B_i \left( \alpha_{i1} \prod_{u=l+1}^{t} \alpha_{iu}^{\widetilde{a_1}^{(u)}/d_u} \right)^{k_1} \cdots \left( \alpha_{il} \prod_{u=l+1}^{t} \alpha_{iu}^{\widetilde{a_l}^{(u)}/d_u} \right)^{k_l} \\
			&= B_i \beta_{i1}^{k_1} \cdots \beta_{il}^{k_l}
		\end{align*}
		with
		\begin{equation*}
			\beta_{ij} := \alpha_{ij} \prod_{u=l+1}^{t} \alpha_{iu}^{\widetilde{a_j}^{(u)}/d_u}
		\end{equation*}
		for $ j = 1,\ldots,l $ and $ i = 1,\ldots,\widetilde{q} $.
		As the $ \alpha_{ij} $ are units in the ring of integers, by our construction of the number field $ L $ the $ \beta_{ij} $ are units in the ring of integers of $ L $.
		
		Therefore we have
		\begin{equation*}
			\sigma_i(\varepsilon_1)^{h_1} \cdots \sigma_i(\varepsilon_r)^{h_r} = B_i \beta_{i1}^{k_1} \cdots \beta_{il}^{k_l}
		\end{equation*}
		and by division
		\begin{equation}
			\label{p7-eq:betatransform}
			\sigma_i(\varepsilon_1)^{h_1 - \widehat{h_1}} \cdots \sigma_i(\varepsilon_r)^{h_r - \widehat{h_r}} = \beta_{i1}^{k_1 - \widehat{k_1}} \cdots \beta_{il}^{k_l - \widehat{k_l}}
		\end{equation}
		for $ i = 1,\ldots,\widetilde{q} $.
		Since $ \beta_{ij} $ is a unit, also $ \sigma_i^{-1}(\beta_{ij}) $ is a unit and by the Dirichlet unit theorem we can write this as
		\begin{equation*}
			\sigma_i^{-1}(\beta_{ij}) = \zeta^{(ij)} \varepsilon_1^{w_1^{(ij)}} \cdots \varepsilon_r^{w_r^{(ij)}}
		\end{equation*}
		for rational integers $ w_1^{(ij)},\ldots,w_r^{(ij)} $ and a root of unity $ \zeta^{(ij)} $.
		Applying $ \sigma_i $ yields
		\begin{equation*}
			\beta_{ij} = \sigma_i(\zeta^{(ij)}) \sigma_i(\varepsilon_1)^{w_1^{(ij)}} \cdots \sigma_i(\varepsilon_r)^{w_r^{(ij)}}.
		\end{equation*}
		Now we put this into equation \eqref{p7-eq:betatransform} and apply $ \sigma_i^{-1} $ to get
		\begin{align*}
			\varepsilon_1^{h_1 - \widehat{h_1}} \cdots \varepsilon_r^{h_r - \widehat{h_r}} &= \left( \zeta^{(i1)} \varepsilon_1^{w_1^{(i1)}} \cdots \varepsilon_r^{w_r^{(i1)}} \right)^{k_1 - \widehat{k_1}} \cdots \left( \zeta^{(il)} \varepsilon_1^{w_1^{(il)}} \cdots \varepsilon_r^{w_r^{(il)}} \right)^{k_l - \widehat{k_l}} \\
			&= \zeta^{(i1)^{k_1 - \widehat{k_1}}} \cdots \zeta^{(il)^{k_l - \widehat{k_l}}} \prod_{v=1}^{r} \varepsilon_v^{(k_1 - \widehat{k_1})w_v^{(i1)} + \cdots + (k_l - \widehat{k_l})w_v^{(il)}}.
		\end{align*}
		Since the representation of any unit in the Dirichlet unit theorem is uniquely determined, we get
		\begin{align*}
			1 &= \zeta^{(i1)^{k_1 - \widehat{k_1}}} \cdots \zeta^{(il)^{k_l - \widehat{k_l}}} \\
			h_1 - \widehat{h_1} &= (k_1 - \widehat{k_1})w_1^{(i1)} + \cdots + (k_l - \widehat{k_l})w_1^{(il)} \\
			&\vdots \\
			h_r - \widehat{h_r} &= (k_1 - \widehat{k_1})w_r^{(i1)} + \cdots + (k_l - \widehat{k_l})w_r^{(il)}.
		\end{align*}
		We rewrite this in matrix notation which results in
		\begin{align*}
			\begin{pmatrix}
				h_1 - \widehat{h_1} \\
				h_2 - \widehat{h_2} \\
				\vdots \\
				h_r - \widehat{h_r}
			\end{pmatrix}
			&=
			\begin{pmatrix}
				w_1^{(i1)} & w_1^{(i2)} & \cdots & w_1^{(il)} \\
				w_2^{(i1)} & w_2^{(i2)} & \cdots & w_2^{(il)} \\
				\vdots & \vdots & \ddots & \vdots \\
				w_r^{(i1)} & w_r^{(i2)} & \cdots & w_r^{(il)}
			\end{pmatrix}
			\begin{pmatrix}
				k_1 - \widehat{k_1} \\
				k_2 - \widehat{k_2} \\
				\vdots \\
				k_l - \widehat{k_l}
			\end{pmatrix} \\
			&=
			\begin{pmatrix}
				w_1^{(i1)} & w_1^{(i2)} & \cdots & w_1^{(il)} & 0 & \cdots & 0 \\
				w_2^{(i1)} & w_2^{(i2)} & \cdots & w_2^{(il)} & 0 & \cdots & 0 \\
				\vdots & \vdots & \ddots & \vdots & \vdots & \ddots & \vdots \\
				w_r^{(i1)} & w_r^{(i2)} & \cdots & w_r^{(il)} & 0 & \cdots & 0
			\end{pmatrix}
			\begin{pmatrix}
				k_1 - \widehat{k_1} \\
				k_2 - \widehat{k_2} \\
				\vdots \\
				k_s - \widehat{k_s}
			\end{pmatrix}.
		\end{align*}
		Now we transpose the equation and get
		\begin{equation}
			\label{p7-eq:diffeq}
			\left( h_1 - \widehat{h_1}, h_2 - \widehat{h_2}, \ldots, h_r - \widehat{h_r} \right) = \left( k_1 - \widehat{k_1}, k_2 - \widehat{k_2}, \ldots, k_s - \widehat{k_s} \right) \cdot A^{(i)}
		\end{equation}
		with
		\begin{equation*}
			A^{(i)} =
			\begin{pmatrix}
				w_1^{(i1)} & w_2^{(i1)} & \cdots & w_r^{(i1)} \\
				w_1^{(i2)} & w_2^{(i2)} & \cdots & w_r^{(i2)} \\
				\vdots & \vdots & \ddots & \vdots \\
				w_1^{(il)} & w_2^{(il)} & \cdots & w_r^{(il)} \\
				0 & 0 & \cdots & 0 \\
				\vdots & \vdots & \ddots & \vdots \\
				0 & 0 & \cdots & 0
			\end{pmatrix}
			\in \ZZ^{s \times r}
		\end{equation*}
		for $ i = 1,\ldots,\widetilde{q} $.
		Since the $ \ \widehat{}\ $-vector is a fixed element of our sequence (namely the first one) we can define the vector $ b^{(i)} \in \ZZ^r $ to be the solution of
		\begin{equation}
			\label{p7-eq:initeq}
			\left( \widehat{h_1}, \widehat{h_2}, \ldots, \widehat{h_r} \right) = \left( \widehat{k_1}, \widehat{k_2}, \ldots, \widehat{k_s} \right) \cdot A^{(i)} + b^{(i)}
		\end{equation}
		for $ i = 1,\ldots,\widetilde{q} $.
		Adding the equations \eqref{p7-eq:diffeq} and \eqref{p7-eq:initeq} yields
		\begin{equation}
			\label{p7-eq:ilattice}
			\left( h_1, h_2, \ldots, h_r \right) = \left( k_1, k_2, \ldots, k_s \right) \cdot A^{(i)} + b^{(i)}
		\end{equation}
		for $ i = 1,\ldots,\widetilde{q} $.
		
		In the next step we consider two instances of equation \eqref{p7-eq:ilattice} for different indices $ i_1 $ and $ i_2 $.
		Subtracting one of them from the other one gives
		\begin{equation*}
			0 = \left( k_1, k_2, \ldots, k_s \right) \cdot \left( A^{(i_1)} - A^{(i_2)} \right) + b^{(i_1)} - b^{(i_2)}.
		\end{equation*}
		Since we have excluded any further linear dependencies of $ k_1,\ldots,k_l $ in the paragraph containing equation \eqref{p7-eq:krepresentation}, it must hold that $ A^{(i_1)} = A^{(i_2)} $ and $ b^{(i_1)} = b^{(i_2)} $.
		Therefore we can omit the superscript and write
		\begin{equation*}
			A := A^{(1)} = A^{(2)} = \cdots = A^{(\widetilde{q})}
		\end{equation*}
		as well as
		\begin{equation*}
			b := b^{(1)} = b^{(2)} = \cdots = b^{(\widetilde{q})}
		\end{equation*}
		in what follows.
		
		We will now evaluate the multi-recurrence $ H\red $ at the shifted sublattice given by $ A $ and $ b $. For each summand we get the identity
		\begin{align*}
			\widetilde{\tau_i} &\sigma_i(\varepsilon_1)^{h_1} \cdots \sigma_i(\varepsilon_r)^{h_r}\left|_{(h_1,\ldots,h_r) = (k_1,\ldots,k_s)A+b} \right. = \\
			&= \widetilde{q_i} \prod_{v=1}^{r} \sigma_i(\varepsilon_v)^{w_v^{(i1)} k_1 + \cdots + w_v^{(il)} k_l} \\
			&= \widetilde{q_i} \left( \sigma_i\left( \zeta^{(i1)^{k_1}} \cdots \zeta^{(il)^{k_l}} \right) \right)^{-1} \prod_{u=1}^{l} \left( \sigma_i(\zeta^{(iu)}) \sigma_i(\varepsilon_1)^{w_1^{(iu)}} \cdots \sigma_i(\varepsilon_r)^{w_r^{(iu)}} \right)^{k_u} \\
			&= \widetilde{q_i} \left( \sigma_i\left( \zeta^{(i1)^{k_1}} \cdots \zeta^{(il)^{k_l}} \right) \right)^{-1} \beta_{i1}^{k_1} \cdots \beta_{il}^{k_l} \\
			&= q_i \beta_{i1}^{k_1} \cdots \beta_{il}^{k_l},
		\end{align*}
		where the last equality holds (only) for $ (k_1,\ldots,k_s) $ within an arithmetic progression of $ s $-dimensional vectors. This arithmetic progression can be chosen in such a way that it contains infinitely many of our vectors and that it is the same for all summands, i.e. for $ i = 1,\ldots, \widetilde{n} $.
		Thus, along an arithmetic progression we have the identity
		\begin{equation*}
			H\red(h_1,\ldots,h_r)\left|_{(h_1,\ldots,h_r) = (k_1,\ldots,k_s)A+b} \right. = \sum_{i=1}^{\widetilde{n}} q_i \beta_{i1}^{k_1} \cdots \beta_{il}^{k_l} =: G^*(k_1,\ldots,k_s).
		\end{equation*}
		Since we have seen above that $ H\red = G\red $ for the vectors in our sequence, $ G_0^{(II)} := G\red - G^* $ is zero for the vectors in our sequence.
		
		In the sequel we will use the shortcut $ \left|_{\sharp} \right. $ for $ \left|_{(h_1,\ldots,h_r) = (k_1,\ldots,k_s)A+b} \right. $ to make the chain of equalities more readable.
		Moreover, we define
		\begin{equation*}
			G_0 := G_0^{(I)} + G_0^{(II)} - H_0\left|_{\sharp} \right..
		\end{equation*}
		We emphasize that $ G_0 $ is zero for the vectors in our sequence.
		Putting all things together we get
		\begin{align*}
			G &= G\red + G_0^{(I)} = G^* + G_0^{(II)} + G_0^{(I)} \\
			&= H\red\left|_{\sharp} \right. + G_0^{(II)} + G_0^{(I)} \\
			&= H\left|_{\sharp} \right. - H_0\left|_{\sharp} \right. + G_0^{(II)} + G_0^{(I)} \\
			&= H\left|_{\sharp} \right. + G_0
		\end{align*}
		as an identity along an arithmetic progression.
		Thus $ G $ has the form of an unavoidable exception.
	\end{proof}
	
	It remains to prove Theorem \ref{p7-thm:specialcase}. Since the procedure is the same as in the proof of Theorem \ref{p7-thm:generalcase} we will only describe the differences.
	
	\begin{proof}[Proof of Theorem \ref{p7-thm:specialcase}]
		We have $ s = 1 $.
		Assume that for some $ \ell \in \set{1,\ldots,n} $ there are infinitely many values of $ x_{\ell} $ such that $ (x_1,\ldots,x_n) \in \ZZ^n $ is a solution of \eqref{p7-eq:normformeq} and
		\begin{equation*}
			x_{\ell} = G(k_1)
		\end{equation*}
		for suitable $ k_1 \in \NN $.
		Then there are obviously no equations of the form $ a^{(1)} k_1 = b^{(1)} $ valid for infinitely many $ k_1 $ unless $ a^{(1)} = 0 $.
		Therefore in the construction in the previous proof we get $ L = K $.
		
		In the new equation \eqref{p7-eq:alphaisunit} we have on the right hand side the expression
		\begin{equation*}
			\alpha_{i1}^{k_1 - \widehat{k_1}}.
		\end{equation*}
		Thus we can deduce that $ \alpha_{i1} $ must be an algebraic integer. It is not necessary to assume this.
		
		In the same way as in the proof of Theorem \ref{p7-thm:generalcase} we get
		\begin{equation*}
			G = H\left|_{\sharp} \right. + G_0
		\end{equation*}
		along an arithmetic progression.
		Since $ L = K $ the recurrence $ H $ comes directly from the solutions of \eqref{p7-eq:normformeq}.
		Since $ G_0 $ has infinitely many zeros, by the Skolem-Mahler-Lech theorem we have $ G_0 = 0 $ if we go to a new arithmetic progression.
		Thus $ G $ has the form of a reduced unavoidable exception.
	\end{proof}

\end{document}